\documentclass{amsart}

\usepackage{amsfonts, amsmath,amscd, amssymb, latexsym, mathrsfs, stmaryrd, verbatim, wasysym }
\usepackage{graphicx}
\usepackage[all]{xy}

\newtheorem{theorem}{Theorem}
\newtheorem{definition}{Definition}

\newtheorem{corollary}[theorem]{Corollary}

\newtheorem{lemma}[theorem]{Lemma}
\newtheorem{proposition}[theorem]{Proposition}

\numberwithin{equation}{section}
\numberwithin{theorem}{section}

 \newcommand\datver[1]{\def\datverp%
 {\par\boxed{\boxed{\text{Version: #1; Run: \today}}}}}
 
 \usepackage{color}
\definecolor{darkgreen}{cmyk}{1,0,1,.2}
\definecolor{m}{rgb}{1,0.1,1}

\renewcommand{\bar}{\overline}

\renewcommand{\hat}[1]{\widehat{#1}}

\newcommand{\script}[1]{\textsc{#1}}

\newcommand{\wt}[1]{\widetilde{#1}}

\newcommand\eps\varepsilon

\newcommand\pa{\partial}

\newcommand\ie{\operatorname{ie}}
\newcommand\iie{\operatorname{iie}}

\newcommand\Ie{{}^{\ie}}
\newcommand\Iie{{}^{\iie}}

\newcommand \C {\mathbb{C}}

\newcommand\ch{\operatorname{ch}}
\newcommand\CI{{\mathcal{C}}^{\infty}}

\newcommand{\lrpar}[1]{\left( #1 \right)}

\DeclareMathOperator*{\btimes}{\times} 

\newcommand{\Ch}{\operatorname{Ch}}

\newcommand{\Dom}{\operatorname{Dom}}

\newcommand\id{\operatorname{id}}
\newcommand\Id{\operatorname{Id}}

\newcommand{\Ind}{\operatorname{Ind}}

\newcommand{\sign}{\operatorname{sign}}

\newcommand\Mand{\text{ and }}

\DeclareMathAlphabet{\mathpzc}{OT1}{pzc}{m}{it}

\newcommand\paperintro%
        {%
         }
\newcommand\paperbody%
        {%
         }

\newcommand\bbC{\mathbb{C}}
\newcommand\bbD{\mathbb{D}}

\newcommand\bbK{\mathbb{K}}

\newcommand\bbN{\mathbb{N}}

\newcommand\bbQ{\mathbb{Q}}
\newcommand\bbR{\mathbb{R}}
\newcommand\bbS{\mathbb{S}}

\newcommand\bbZ{\mathbb{Z}}

\newcommand\cM{\mathcal{M}}
\newcommand\cN{\mathcal{N}}

\newcommand\cS{\mathcal{S}}
\newcommand\cT{\mathcal{T}}
\newcommand\cU{\mathcal{U}}
\newcommand\cV{\mathcal{V}}

\datver{\Modified}

\begin{document}
\title
[The signature package on Witt spaces, II.]
{The signature package on Witt spaces, II. \\Higher signatures. }

\author{Pierre Albin}
\address{Courant Institute and Institute for Advanced Study }
\email{albin@courant.nyu.edu}
\author{Eric Leichtnam}
\address{CNRS Institut de Math\'ematiques de Jussieu}
\author{Rafe Mazzeo}
\address{Department of Mathematics, Stanford University}
\email{mazzeo@math.stanford.edu}
\author{Paolo Piazza}
\address{Dipartimento di Matematica, Sapienza Universit\`a di Roma}
\email{piazza@mat.uniroma1.it}


\begin{abstract}
This is a sequel to the paper ``The signature package on Witt spaces, I.  Index classes'' by the same authors.
In the first part we investigated, via a parametrix construction, the regularity properties of the signature operator on a stratified Witt pseudomanifold, proving, in particular,  that one can define a K-homology 
signature class. We also established the existence of an analytic index class for the signature operator twisted by a $C^*_r\Gamma$ Mischenko bundle  and proved that the K-homology signature class
is mapped to the signature index class by the assembly map. In this paper we continue our study, showing  that the signature index class is invariant under rational Witt bordisms and stratified homotopies. We are also able
to identify this analytic class with the topological analogue of the Mischenko symmetric signature recently defined by Banagl. Finally, we define Witt-Novikov  higher signatures and show that our analytic results imply a purely topological theorem, namely that the Witt-Novikov higher signatures  are stratified homotopy invariants if the assembly map in K-theory is rationally injective.
\end{abstract}

\maketitle

\tableofcontents

\section{Introduction}

This is the second of two papers devoted to  what we have called the signature package
on Witt spaces. For the benefit of the reader, we
restate  the statement of the signature package \footnote{In this second paper we have slightly modified the
statement of the signature package, giving more weight to the analytic proof of the (stratified)
homotopy invariance of the signature index class}.

Let $\hat X$ be an orientable Witt pseudomanifold with fundamental group $\Gamma$. 
We endow the regular part of $\hat X$ with an iterated conic metric $g$.
Let $\hat X ^\prime$ 
be a Galois $\Gamma$-covering over $\hat X$ and $r: \hat X\to B\Gamma$ a classifying map for 
 $\hat X ^\prime$ . 
The signature package for the pair $( \hat X  ,r: \hat X  \to B\Gamma)$ is the following list of results:

\begin{enumerate}
\item the signature operator defined by the iterated conic metric $g$
with values in the Mischenko bundle $r^* E\Gamma \times_\Gamma C^*_r\Gamma$
defines  a signature index class  $\Ind (\wt \eth_{\sign})\in K_* (C^*_r \Gamma)$, $* \equiv \dim  \hat X  \;{\rm mod}\; 2$; 
\item the signature index class is a rational Witt bordism invariant; more precisely  it defines a group homomorphism 
$\Omega^{{\rm Witt}}_* (B\Gamma) \to K_* (C^*_r \Gamma)\otimes \bbQ$;
\item the signature 
index class is a stratified homotopy invariant; 
\item  there is a  K-homology signature class $[\eth_{\sign}]\in K_* ( \hat X )$ whose Chern
 character is, rationally, the  homology L-Class of a Witt space;
\item the assembly map $\beta: K_* (B\Gamma)\to K_* (C^*_r\Gamma)$ sends the
class $ r_* [\eth_{\sign}]$ into $\Ind (\wt \eth_{\sign})$;
\item using the homology L-class one can define Witt-Novikov  higher signatures and 
if  the assembly map  is rationally injective these are 
stratified homotopy invariants.
\end{enumerate}

We label the {\em full} signature package the one decorated by the following item

\medskip
 \begin{list}
 {(7)} \item There is a topologically defined $C^*$-algebraic symmetric signature 
 $\sigma^{{\rm Witt}}_{C^*_r \Gamma}( \hat X ,r)$\\$\in$ $  K_* (C^*_r \Gamma)$; moreover 
 $\sigma^{{\rm Witt}}_{C^*_r \Gamma}( \hat X ,r)=\Ind (\wt \eth_{\sign})$ rationally.
   \end{list}

\medskip
Our first paper was devoted to a detailed study of the analytic properties of the signature operator
defined by an iterated conic metric on a Witt space.
The results proved there establish, in particular, item 1), the first part of item 4) (i.e. the existence of a well defined K-homology class)
as well as item 5). See the next Section for precise statements.
In this second paper we will establish the rest of the (full) signature package.
We shall use freely the notations of  part I.

\bigskip
\noindent
{\bf Acknowledgements.} 

P.A., R.M. and P.P. all wish to thank MSRI for hospitality and financial support during the Fall Semester of 2008
when much of the latter stages of this paper were completed. P.A. was partly supported by NSF grant DMS-0635607002 and by an NSF Postdoctoral 
Fellowship, and wishes to thank Stanford for support during visits; R.M. was supported by NSF grants 
DMS-0505709 and DMS-0805529, and enjoyed the hospitality and financial support of MIT, 
Sapienza Universit\`a di Roma and the Beijing International Center for Mathematical Research;
P.P. wishes to thank the CNRS and  Universit\'e Paris 6 for financial support during  visits to Institut 
de Math\'ematiques de Jussieu in Paris. E.L was  partially supported during visits to 
Sapienza Universit\`a di Roma by CNRS-INDAM (through the bilateral agreement GENCO (Non commutative 
Geometry)) and the Italian {\it Ministero dell' Universit\`a  e della  Ricerca Scientifica} (through the 
project "Spazi di moduli e teoria di Lie").

We thank Markus Banagl for many interesting discussions on the
notion of symmetric signature on Witt spaces; we thank Shmuel Weinberger for suggesting the
proof of Proposition \ref{prop:magic}.

\section{K-theory classes defined by the signature operator}
 
  In this Section  we recall those results of  Part 1 that are directly involved in the signature
package.
 First of all, we have a K-homology signature class:
 
\begin{theorem}\label{theo:k-homology}
The signature operator $\eth_{\sign}$ associated to a Witt space 
$\widehat{X}$ endowed with an iterated conic metric $g$ defines an unbounded Fredholm module
for $C(\widehat{X})$ and thus a class   $[\eth_{\sign}]\in KK_* (C(\widehat{X}),\bbC)$, $* \equiv\dim X \,{\rm mod} \,2$.
Moreover, the class  $[\eth_{\sign}]$ does not depend on the choice of iterated conic metric on $\widehat{X}$.
\end{theorem}

Assume now that we are also given a continuous map $r:\widehat{X}\to B\Gamma$ and let
$\Gamma \to \widehat{X}^\prime \to \widehat{X}$ the Galois $\Gamma$-cover induced by $E\Gamma\to B\Gamma$. 
We consider the Mishchenko bundle
\begin{equation*}
	\wt{C^*_r}\Gamma: = C^*_r\Gamma \btimes_\Gamma \widehat{X}^\prime.
\end{equation*} 
and the signature operator with values in the restriction of $\wt{C^*_r}\Gamma$ to $X$, which we denote
by $ \wt \eth_{\sign}$. Our second result is Part 1 was the following theorem together with its corollary

\begin{theorem}\label{theo:kk}
The twisted signature operator $\wt \eth_{\sign}$ and the $C^*_r\Gamma$-Hilbert
module $L^2_{\iie,\Gamma}(X;\Iie\Lambda_\Gamma^*X)$ define an unbounded 
Kasparov $(\bbC,C^*_r\Gamma)$-bimodule and thus a class in $KK_* (\bbC, C^*_r \Gamma)=K_* (C^*_r\Gamma)$.
We call this the index class associated to $\wt \eth_{\sign}$ and denote it by $\Ind (\wt \eth_{\sign})\in K_* (C^*_r\Gamma)$.\\
Moreover, if  we denote by $[[\eth_{\sign}]]\in KK_*(C(\widehat{X})\otimes C^*_r\Gamma, C^*_r\Gamma)$ 
the class obtained from $[\eth_{\sign}]\in KK_*(C(\widehat{X}),\bbC)$ by tensoring with  $C^*_r\Gamma$, then
$\Ind (\wt \eth_{\sign})$ is equal to the Kasparov product of the class defined by Mishchenko bundle 
$[\widetilde{C^*_r}\Gamma]\in  KK_0(\bbC,C(\widehat{X})\otimes C^*_r\Gamma)$ with  $[[\eth_{\sign}]]$:
\begin{equation}\label{tensor}
\Ind (\wt \eth_{\sign})= [\widetilde{C^*_r}\Gamma]\otimes [[\eth_{\sign}]]
\end{equation}
In particular, the index class $\Ind (\wt \eth_{\sign})$ does not depend on the choice of the iterated conic metric.
\end{theorem}

\begin{corollary}\label{cor:assembly}
Let $\beta:K_* (B\Gamma)\to K_* (C^*_r\Gamma)$ be the assembly map and $r_* [\eth_{\sign}]\in K_* (B\Gamma)$
the push-forward of the signature K-homology class, then 
\begin{equation}\label{assembly}
\beta(r_* [\eth_{\sign}])=\Ind (\wt \eth_{\sign}) \text{ in }  K_* (C^*_r\Gamma).
\end{equation}
\end{corollary}

For the definition of the assembly map $\beta$ see, for example,  \cite{Kasparov-contemporary} (p. 81).

\section{Witt bordism invariance}

Let $\hat Y$ be an oriented odd dimensional Witt space  with boundary $\partial \hat Y= \hat X.$  
We assume that $\hat  Y$ is a smoothly stratified space having a product structure 
near its boundary.
We endow $\hat Y$ with a conic iterated metric having a product structure near $\partial \hat Y= \hat X$  
and inducing a conic iterated metric metric $g$ on $ \hat X.$ Consider 
a reference map $r: \hat Y \rightarrow B \Gamma,$ its restriction to $\hat X$ and $g$ induce a  $C^*_r \Gamma-$linear signature operator on 
$\hat X$. In this Section only we shall be very precise and denote this operator by $\wt \eth_{\sign}( \hat X)$.

\begin{theorem} \label{thm:bordism} We have $ \Ind \wt \eth_{\sign} (\hat X) =0 $ in $K_* (C^*_r \Gamma) \otimes_\bbZ \bbQ.$
\end{theorem}
\begin{proof} We follow \cite[Section 4.3]{LP-Cut} and Higson \cite[Theorem 5.1]{Higson}. 
Denote by ${\hat Y}^\prime \rightarrow \hat Y$ and ${ \hat X}^\prime \rightarrow \hat X$ the two $\Gamma-$coverings
associated to the reference map $r: \hat Y \rightarrow B \Gamma.$

The analysis of \cite{ALMP1} shows that the operator $\wt \eth_{\sign}( \hat X) $ induces a class $[\wt \eth_{\sign}( \hat X) ]$ in the  
Kasparov group $KK^0 (C_0( \partial {\hat Y}), \C_r^* \Gamma).$ 
In terms of the constant map $\pi^{\partial {\hat Y}}: \partial \hat Y \rightarrow \{ {\rm pt }\},$ one has:
$$
{ \Ind} \, \wt \eth_{\sign} (\hat X ) = \pi^{\partial \hat Y}_*( [\wt \eth_{\sign} ( \hat X )])  \in 
KK^0( \C , \C_r^* \Gamma) \simeq K^0( C^*_r \Gamma).
$$
Now let $C_{ {\partial \hat Y}}({\hat Y}) \subset C({\hat Y})$ denote 
the ideal of continuous functions on $\hat Y$ vanishing on the boundary. Let $i: \partial {\hat Y} \rightarrow {\hat Y}$
denote the inclusion and consider the long exact sequence in $KK (\cdot , \C_r^* \Gamma)$ associated 
to the semisplit short exact sequence:

\begin{equation}
0 \rightarrow C_{ \partial {\hat Y}}({\hat Y}) \overset{j}{\rightarrow} C({\hat Y}) \overset{q}{\rightarrow} C(\partial {\hat Y}) \rightarrow 0
\end{equation}
(see Blackadar \cite[page 197]{Blackadar}). In particular, we have the exactness of
$$
KK^1 (C_{ \partial {\hat Y}}({\hat Y}) , \C^*_r \Gamma) \overset{\delta}{\rightarrow} KK^0 ( C(\partial {\hat Y}), \C^*_r \Gamma) \overset{i_*}{\rightarrow} KK^0 ( C( { \hat Y}), \C^*_r \Gamma)
$$ and thus $i_* \circ \delta=0.$
Recall that the conic iterated metric on $\hat Y$ (with product structure near $\partial \hat Y= \hat X$) allows us 
to define a $C^*_r \Gamma-$linear twisted signature operator $\wt \eth_{\sign}$ on $\hat Y$ with 
coefficients in the bundle ${\hat Y}^\prime \times_\Gamma C^*_r \Gamma \rightarrow \hat Y.$ This twisted signature operator
allows us to define a class $[\wt \eth_{\sign}] \in KK^1(C_{\partial {\hat Y}}({\hat Y}), \C_r^*\Gamma).$
\begin{lemma} One has $\delta  [\wt \eth_{\sign}]= [2 \wt \eth_{\sign} (\hat X ) ].$
\end{lemma}
\begin{proof} We are using  the proof
of Theorem 5.1 of Higson \cite{Higson}. We can replace ${\hat Y}$ by a 
 collar neighborhood $W$ ($\simeq [0,1[ \times \partial {\hat Y}$).
Consider the differential operator $d:$
$$
d=\begin{pmatrix}
	0&  - i \frac{d}{d x }  \\
	- i \frac{d}{d x }  & 0
	\end{pmatrix}
$$ acting on $[0,1].$ It defines a class in $KK^1 ( C_0 (0,1) , \C^*_r \Gamma).$  Recall that the Kasparov product
$[d]\otimes \cdot $ induces an isomorphism:
$$
[d]\otimes \cdot : KK^0(C(\partial {\hat Y}), \C^*_r \Gamma) \rightarrow 
KK^1(C_{\partial {\hat Y}}(W), \C^*_r \Gamma).
$$
As in \cite{Higson}, the connecting map $\delta:$
$$
KK^1(C_{\partial {\hat Y}}(W), \C^*_r \Gamma) \overset{\delta}{\rightarrow} KK^0(C(\partial {\hat Y}), \C^*_r \Gamma)
$$ is given by the inverse of $[d]\otimes \cdot $. Denote by $\mathcal{D}_{W}$ the restriction 
of $\wt \eth_{\sign}$ to $W$ and recall that $\hat X= \partial \hat Y$ is even dimensional. Then one checks (using \cite{Higson} and \cite[page 296]{PRW})   that the $KK-$class 
$[\mathcal{D}_{W}]$ is equal to $[d]\otimes 2 [\wt \eth_{\sign} (\hat X) ],$ and one finds that 
$ \delta [\mathcal{D}_{W} ] = 2 [\wt \eth_{\sign} (\hat X) ]$ which proves the result.
\end{proof}
Let $\pi^{ \hat Y}: \hat Y \rightarrow \{ {\rm pt }\}$ denote the constant map.  By functoriality, one has:
$$
\pi_*^{\partial \hat Y} = \pi_*^{\hat Y}  \circ i_*.
$$ Since $i_* \circ \delta=0$, the previous Lemma implies that:
$$
2 \Ind \wt \eth_{\sign} (\hat X)= \pi_*^{\partial {\hat Y} }(  [2 \wt \eth_{\sign} (\hat X) ] )=
\pi_*^{ \partial \hat Y}  \circ \delta ( [\mathcal{D}_{W} ]  )= 
\pi_*^{\hat Y} \circ  i_* \circ \delta ( [\mathcal{D}_{W} ]  )=0.
$$
Therefore, Theorem \ref{thm:bordism} is proved.
\end{proof}

\bigskip

We shall denote by $\Omega^{{\rm Witt, s}}_* (B\Gamma)$ the bordism group in the category of 
{\it smoothly} stratified oriented Witt spaces. This group is
generated by the elements  of the form
$[\hat X,r: \hat X\to B\Gamma]$ where $[\hat X,r: \hat X\to B\Gamma]$ is equivalent to the zero element  
if $\hat X$ is the boundary of a smoothly stratified Witt oriented space $\hat Y$  (as in Theorem \ref{thm:bordism}) such that the map $r$
extends continuously to $\hat Y.$
It follows that the index map
\begin{equation}\label{bordsim-inv}
\Omega^{{\rm Witt,s}}_* (B\Gamma)\to K_*(C^*_r \Gamma)\otimes \bbQ,
\end{equation}
sending $[\hat X,r: \hat X\to B\Gamma]\in \Omega^{{\rm Witt,s}}_* (B\Gamma)$ to the higher index class 
$\Ind (\wt \eth_{\sign})$ (for the twisting bundle $ r^* E\Gamma\times_\Gamma C^*_r \Gamma$),
is well defined. As in the closed case, see \cite{ros-weinberget-at-2}, it might be possible to refine this result 
and show that the index map actually defines a group homomorphism $\Omega^{{\rm Witt,s}}_* (B\Gamma)\to K_*(C^*_r \Gamma)$

Recall that Siegel's Witt-bordism groups $\Omega^{{\rm Witt}}_* (B\Gamma)$ are given in terms of equivalence classes
of pairs $(\hat X,u: \hat X\to B\Gamma)$, 
with $\hat X$ a Witt space  which is not necessarily {\it smoothly} stratified.

We also 
recall that, working with PL spaces, Sullivan \cite{Sullivan} has defined 
the notion of connected KO-Homology $ko_*$ (see also \cite[page 1069]{Se}).
Siegel \cite[Chapter 4]{Se} has shown that the natural map 
$\Omega^{{\rm SO}}_*(B \Gamma) \otimes_\bbZ \bbQ \rightarrow  
\Omega^{{\rm Witt}}_* (B\Gamma) 
\otimes_\bbZ \bbQ $ is surjective by showing that the natural map $\Omega^{{\rm SO}}_* (B \Gamma) \otimes_\bbZ \bbQ \rightarrow ko_* (B \Gamma) \otimes_\bbZ \bbQ$ is surjective and the  Siegel's natural map (\cite{Se})  $\Omega^{{\rm Witt}}_* (B\Gamma) \otimes_\bbZ \bbQ  
\rightarrow ko_* (B \Gamma) \otimes_\bbZ \bbQ$ is an isomorphism.
We need to extend these results  for the corresponding groups associated with the category of smoothly stratified spaces. 
\begin{proposition} \label{prop:ko}
The natural map $\Omega^{{\rm SO}}_* (B \Gamma) \otimes_\bbZ \bbQ \rightarrow \Omega^{{\rm Witt,s}}_* (B\Gamma) 
\otimes_\bbZ \bbQ$ is surjective.
\end{proposition}
\begin{proof} 
Theorem 4.4 of \cite{Se} is still valid (by inspection) if one works in the category 
of smoothly stratified oriented Witt spaces. Namely, if $\hat X$ is an irreducible smoothly stratified Witt space 
of even dimension $>0$ such that its signature $w(\hat X)=0$, then $\hat X$ is Witt cobordant to 
zero in the category of smoothly stratified Witt spaces. The arguments of \cite[Chapter 4]{Se} 
show that the Siegel's natural map:
$$
\Omega^{{\rm Witt,s}}_* (B\Gamma) \otimes_\bbZ \bbQ  
\rightarrow ko_* (B \Gamma) \otimes_\bbZ \bbQ
$$ is an isomorphism and, using the surjectivity of the natural map
$$\Omega^{{\rm SO}}_* (B \Gamma) \otimes_\bbZ \bbQ \rightarrow ko_* (B\Gamma) 
\otimes_\bbZ \bbQ,$$
 one gets the Proposition.
\end{proof}

\section{The homology L-class of a Witt space. Higher signatures.} \label{sec:Lclass}

The homology $L$-class $L_*(\hat X)\in H_* (\hat X,\bbQ)$ of a Witt space
$\hat X$ was defined independently by Goresky and MacPherson \cite{GM},
following ideas of Thom \cite{Thom}, and by Cheeger \cite{Ch}. See also Siegel \cite{Se}.
In this paper we shall adopt the approach of Goresky and
MacPherson.
%
We briefly recall the definition: if $\hat X$ has dimension
$n$, $k\in \bbN$ is such that $2k-1>n$, and $\cN$ denotes the `north pole' of $\bbS^k$, one can show that the map
$\sigma: \pi^k (\hat X)\to \bbZ$ that associates to $[f:\hat X\to S^k]$ the Witt-signature 
of $f^{-1} (\cN)$ is well defined and a group homomorphism. Now, by Serre's theorem,
the Hurewicz map $\pi^k (\hat X)\otimes \bbQ \to H^k (\hat X,\bbQ)$ is an isomorphism 
for $2k-1>n$ and we can thus view the above homomorphism, $\sigma\otimes \Id_{\bbQ}$, as a linear functional in
$\mathrm{Hom} (H^k (\hat X),\bbQ)\simeq H_k (\hat X,\bbQ)$. This defines $L_k (\hat X)\in 
H_k (\hat X,\bbQ)$. The restriction $2k-1>n$ can be removed by crossing with a high dimensional
sphere in the following way. Choose a positive integer $l $ such that 
$2(k+\ell) -1 > n +\ell$ and $k + \ell > n.$
Then by the above construction, $L_{k + \ell} (\hat X\times S^\ell)$ is well defined in $H_{k +\ell} (\hat X\times S^\ell , \bbQ).$ 
Since  $k + \ell > n,$ the Kunneth Theorem shows that there is a natural isomorphism $ I : H_{k +\ell} (\hat X\times S^\ell , \bbQ) 
\rightarrow H_{k } (\hat X , \bbQ).$ One then defines:
$L_k (\hat X):= I ( L_{k+l} (\hat X\times S^\ell) )$.

Once we have a homology $L$-class we can define the higher signatures as follows. 

\begin{definition}\label{def:higher-signatures}
Let $\hat X$ be a Witt space and $\Gamma:=\pi_1 (M)$.
Let $r: \hat X\to B\Gamma$ be a classifying map for the universal cover.
The (Witt-Novikov) higher signatures of $\hat X$ are the collection of rational numbers:
\begin{equation}\label{higher-signatures}
\{<\alpha,r_* L_*( \hat X)>\,,\alpha\in H^* (B\Gamma,\bbQ)\}
\end{equation}
We set $\sigma_\alpha (\hat X):= <\alpha,r_* L_*( \hat X)>$.
\end{definition}

If $X$ is an oriented  closed compact manifold and $r:X\to B\pi_1 (X)$ is the classifying map, it is not difficult
to show that 
$$ <\alpha,r_* L_*(X)>=<L(X)\cup r^* \alpha,[X]>\equiv \int L(X)\cup r^* \alpha\,.$$
Thus our definition is consistent with the usual definition of Novikov higher signatures in the closed case.

The Novikov conjecture in the closed case  is the statement
that all the higher signatures $\{<L(X)\cup r^* \alpha,[X]>\,,\,\alpha\in H^* (B\Gamma,\bbQ)\}$ are homotopy invariants.

The Novikov conjecture in the Witt case 
 is the statement that the Witt-Novikov
 higher signatures  $ \{<\alpha,r_* L_*(\hat X)>\,,\,\alpha\in H^* (B\Gamma,\bbQ)\}$
  are {\bf stratified} homotopy
invariants. Notice that intersection homology is not a homotopy invariant theory;
however, it is a stratified homotopy-invariant theory, see \cite{friedman}.

We shall need to relate the homology $L$-class of Goresky-MacPherson to the signature
class $[\eth_{\sign} ] \in K_*(\hat X)$.

\begin{theorem}\label{theo:cmwu} (Cheeger/Moscovici-Wu)
The topological homology L-class $L_*(\hat X)\in H_* (\hat X,\bbQ)$ is the
image, under the rationalized homology Chern character, of  
the signature K-homology class $[\eth_{\sign} ]_{\bbQ} \in K_*(\hat X)\otimes \bbQ$; in formul\ae
\begin{equation}\label{chern}
\ch_* [\eth_{\sign} ]_{\bbQ} = L_* (\hat X) \quad \text{in} \quad H_* (\hat X,\bbQ).
\end{equation}
\end{theorem}

This result is due to Cheeger, who proved it for
piecewise flat metric of conic type, and to Moscovici-Wu, who gave an alternative argument
valid also for any metric quasi-isometric to such a metric \cite{Ch}, \cite{Mosc-Wu-Witt}.
It is worth pointing out here that our metrics do belong to the class considered in \cite{Mosc-Wu-Witt}.
Notice that Moscovici-Wu prove that the straight Chern
character of $[\eth_{\sign} ]_{\bbQ} \in K_*(\hat X)\otimes \bbQ$ is equal to $L_*(\hat X)\in H_* (\hat X,\bbQ)$;
the straight Chern character has values in Alexander-Spanier homology; the equality with 
$L_*(\hat X)\in H_* (\hat X,\bbQ)$ is obtained using the isomorphism between Alexander-Spanier and singular homology \cite{Mosc-Wu-Witt}.

\section{Stratified homotopy invariance of the index class: the analytic approach}\label{sect:shi}

One key point in all the index theoretic proofs of the Novikov conjecture for closed oriented
manifolds is the one stating the homotopy invariance of the signature index class in $K_* (C^*_r \Gamma)$.
By this we mean that if $r:X\to B\Gamma$ as above, $f: X^\prime\to X$ is a smooth homotopy equivalence and 
$r^\prime:= r\circ f:X^\prime\to B\Gamma$,
 then the index class, in $K_* (C^*_r\Gamma)$, associated to  $\widetilde{\eth}_{\sign}$ (i.e., associated to
 the signature operator on $X$, $\eth_{\sign}$,
twisted by $r^* E\Gamma\times_\Gamma C^*_r\Gamma$) is equal to the one associated to $\widetilde{\eth}_{\sign}^\prime$
(i.e., associated to
 the signature operator on $X^\prime$, $\eth_{\sign}^\prime$,
twisted by $(r^\prime)^* E\Gamma\times_\Gamma C^*_r\Gamma$).
There are two approaches to this fundamental result:
\begin{enumerate}
\item one proves {\it analytically} that $\Ind (\widetilde{\eth}_{\sign} )=\Ind (\widetilde{\eth}_{\sign}^\prime)$ in $K_* (C^*_r \Gamma)$;
\item one proves that the index class is equal to an a priori homotopy invariant, the Mishchenko ($C^*$-algebraic) symmetric signature.
\end{enumerate}

In this section we pursue the first of these approaches. We shall thus
establish the stratified homotopy invariance of the index class on Witt spaces by following ideas from 
  Hilsum-Skandalis \cite{H-S}, where this property is proved for closed compact manifolds. See also \cite{PS}.

\subsection{Hilsum-Skandalis replacement of $f$}

If $X$ and $Y$ are closed Riemannian manifolds, and $f:X \to Y$ is a homotopy equivalence, it need not be the case that pull-back by $f$ induces a bounded operator in $L^2$. Indeed, suppose $f$ is an embedding and $\phi_\eps$ is a function which equals $1$ on the $\eps$ tubular neighborhood of the image of X.
The $L^2$-norm of $\phi_{\eps}$ is bounded by $C \eps^{\mathrm{codim}_YX}$ and hence tends to zero, while $f^* \phi_{\eps} \equiv 1$ on $X$ and so its $L^2$ norm is constant.
Thus the closure of the graph of $f^*$, say over piecewise constant functions, contains an element of the form $(0,1)$, and is not itself the graph of an operator. 

On the other hand, if $f$ is a submersion, and the metric on $X$ is a submersion metric, then $f^*$ clearly does induce a bounded operator on $L^2$. Since the latter property is a quasi-isometry invariant, and any two metrics on $X$ are quasi-isometric, it follows that pull-back by a submersion always induces a bounded operator in $L^2$.

As one is often presented with a homotopy equivalence $f$ and interested in properties of $L^2$ spaces, it is useful to follow Hilsum and Skandalis \cite{H-S} and replace pull-back by $f$ by an operator that is bounded in $L^2$. We refer to this operator as the HS replacement of $f^*$ and denote it $HS(f)$.

Such a map is constructed as follows. Consider a disk bundle $\pi_Y:\bbD_Y \to Y$ and the associated pulled back bundle  $f^* \bbD_Y$ 
by the map $f: X \rightarrow Y.$ Denote by
 $\pi_X: f^*\bbD_Y \to X$ the induced projection. Then   $f$ admits a natural lift $\bbD(f)$ such that
\begin{equation*}
	\xymatrix{ f^*\bbD_Y \ar[r]^{\bbD(f)} \ar[d] & \bbD_Y \ar[d] \\ X \ar[r]^{f} & Y }
\end{equation*}
commutes. Moreover, we consider  a (smooth) map $e:\bbD_Y \to Y$ such that $p=e \circ \bbD(f): f^*  \bbD_Y \to Y$ is a submersion, and a choice of Thom form $\cT$ for $\pi_X$.
The Hilsum-Skandalis replacement of $f^*$ is then the map
\begin{equation*}
	\xymatrix @R=1pt @C=1pt
	{ HS(f) = HS_{\cT, f^*\bbD_Y, \bbD_Y, e}(f):& \CI(Y; \Lambda^*) \ar[r] & \CI(X; \Lambda^*) \\
	& u \ar@{|->}[r] & (\pi_X)_* (\cT \wedge p^*u) }
\end{equation*}
Notice that $HS(f)$ induces a bounded map in $L^2$ because $(e\circ \bbD(f) )^*$ does.

For example, as in \cite{H-S}, one can start with an embedding $j: Y \to \bbR^N$ and a tubular neighborhood $\cU$ of $j(Y)$ such that $j(\zeta) + \bbD \subseteq \cU$,
and then take $\bbD_X = X \times \bbD$, $\bbD_Y = Y \times \bbD$, $\bbD(f) = f \times \id$, and $e(\zeta, v) = \tau(\zeta + v)$ where $\tau: \cU \to Y$ is the projection.
Alternately, one can take $\bbD_Y$ to be the unit ball subbundle of $TY$  and $e(\zeta, v) = \exp_{f(\zeta)}(v)$. 
We will extend the latter approach to stratified manifolds.

In any case, one can show that $HS(f)$ is a suitable replacement for $f^*$. 
Significantly, using $HS(f)$ we will see that the K-theory classes induced by the signature operators of homotopic stratified manifolds coincide.
\subsection{Stratified homotopy equivalences}

Let $\hat X$ and $\hat Y$ denote stratified spaces, $X$ and $Y$ their regular parts, and $\cS(X)$ and $\cS(Y)$ the corresponding sets of strata.
Following \cite{friedman} and \cite[Def. 4.8.1 ff]{Kirwan-Woolf} we say that a map $f:\hat X \to \hat Y$ is 
{\em stratum preserving} if
\begin{equation*}
	S \in \cS(\hat Y) \implies f^{-1}(S) \text{ is a union of strata of $X$}
\end{equation*}
and {\em codimension preserving} if also
\begin{equation*}
	\mathrm{codim} f^{-1}(S) = \mathrm{codim} S.
\end{equation*}
We will say that a map is {\em strongly stratum preserving} if it is both stratum and codimension preserving.

In these references, a {\em stratum-preserving homotopy equivalence} between stratified spaces is a strongly stratum preserving map $f:\hat X \to \hat Y$ such that there exists another strongly stratum preserving map $g: \hat Y \to \hat X$ with both $f \circ g$ and $g \circ f$ homotopic to the appropriate identity maps through strongly stratum preserving maps. It is shown that stratum-preserving homotopy equivalences induce isomorphisms in intersection cohomology.

Notice that the existence of a homotopy equivalence between {\em closed} manifolds implies that the manifolds have the same dimension, so it is natural to impose a condition like strong stratum preserving on stratified homotopy equivalences.

A smooth strongly stratified map lifts, as in \cite[Prop. 3.2]{BHS}, to a smooth map between the resolutions of the stratified spaces $\wt f: \wt X \to \wt Y$ and will map each boundary face of $\wt X$ to a boundary face of $\wt Y$.
Hence it is a b-map (since its differential preserves vector fields tangent to the boundary) and it necessarily sends fibers of a boundary fibration of $\wt X$ to fibers of a boundary fibration of $\wt Y$ (since it comes from a map of the bases of the fibrations).
The latter implies that the induced map on vector fields preserves the bundle of vertical vector fields at each boundary face and so $\wt f$ (and hence $f$) induces maps
\begin{equation*}
	f^*: \CI(Y; \Ie \Lambda^*) \to \CI(X; \Ie \Lambda^*), \Mand
	f^*: \CI(Y; \Iie \Lambda^*) \to \CI(X; \Iie \Lambda^*),
\end{equation*}
though as on a closed manifold these do not necessarily induce bounded maps in $L^2$.

\subsection{Hilsum-Skandalis replacement on complete edge manifolds}

Suppose ${\wt X}$ and ${\wt Y}$ are both manifolds with boundary and boundary fibrations
\begin{equation*}
	\phi_{\wt X}: \pa {\wt X} \to H_{\wt X}, \quad
	\phi_{\wt Y}: \pa {\wt Y} \to H_{\wt Y}.
\end{equation*}

Endow ${\wt Y}$ with a complete edge metric.
Let $\bbD_{\wt Y} \subseteq {}^eT{\wt Y}$ be the edge vector fields on ${\wt Y}$ with pointwise length bounded by $1$, and let $\exp: \bbD_{\wt Y} \to {\wt Y}$ be the exponential map on ${\wt Y}$ with respect to the edge metric.
The space $\bbD_{\wt Y}$ is itself an (open) edge manifold with boundary fibration $\phi_{\bbD_{\wt Y}}: \pa \bbD_{\wt Y} \to \pa {\wt Y} \to H_{\wt Y}$.
Notice that $\exp$ is a b-map that sends fibers of $\phi_{\bbD_{\wt Y}}$ to fibers of $\phi_{\wt Y}$ 	
and hence induces a map
\begin{equation*}
	\exp_*:{}^eT\bbD_{\wt Y} \to {}^eT{\wt Y}
\end{equation*}
which is seen to be surjective.

Let $f: {\wt X} \to {\wt Y}$ be a smooth b-map that sends fibers of $\phi_{\wt X}$ to fibers of $\phi_{\wt Y}$.
Pulling-back the bundle $\bbD_{\wt Y} \to \wt Y$ to $\wt X$ gives a commutative diagram
\begin{equation}\label{PullBackDiag}
	\xymatrix{
	f^*\bbD_{\wt Y} \ar[r]^{\bar f} \ar[d]^{\pi_{\wt X}} & \bbD_{\wt Y} \ar[d]^{\pi_{\wt Y}} \\ \wt X \ar[r]^f & \wt Y }
\end{equation}
and we can use to construct the Hilsum-Skandalis replacement for pull-back by $f$.
Namely, define  $e = \exp: \bbD_{\wt Y} \to {\wt Y}$, let $\cT$ the pull-back by $\bar f$ of a Thom form for $ \bbD_{\wt Y}$, and let 
\begin{equation}\label{def-of-HS}
	HS(f) = (\pi_{\wt X})_* ( \cT \wedge p^* ): \CI( {\wt Y}; \Ie \Lambda^*) \to \CI({\wt X}; \Ie\Lambda^*)
\end{equation}
with $p = e \circ \bbD(f)$. Observe, that $p$ is a proper submersion and hence a fibration.
Then, as above, $HS(f)$ induces a map between the corresponding $L^2$ spaces.

The generalization to manifolds with corners and iterated fibrations structures is straightforward: we just replace the edge tangent bundle with the iterated edge tangent bundle. 
Indeed, it is immediate that if $\bbD_{\wt Y} \subseteq \Ie T\wt Y$ is the set of iterated edge vector fields with pointwise length bounded by one the exponential map $\exp: \bbD_{\wt Y} \to \wt Y$ with respect to a (complete) iterated edge metric induces a map $\exp_*: \Ie T\bbD_{\wt Y} \to \Ie T\wt Y$. That this map is surjective can be checked locally and follows by a simple induction.
Then given a smooth b-map $f: \wt X \to \wt Y$ with the property that, whenever $H \in \cM_1(\wt X)$ is sent to $K \in \cM_1(\wt Y)$, the fibers of the fibration on $H$ are sent to the fibers of the fibration on $K$,
we end up with a map
\begin{equation*}
	HS(f): \CI(\wt Y, \Ie \Lambda^* ) \to \CI(\wt X, \Ie \Lambda^*)
\end{equation*}
that induces a bounded map between the corresponding $L^2_{\ie}$ spaces.


Next, recall that 
\begin{equation*}
	\CI(\wt Y; \Iie \Lambda^1) = \rho_{\wt Y} \CI(\wt Y; \Ie \Lambda^1) 
\end{equation*}
where $\rho_{\wt Y}$ is a total boundary defining function for $\partial \wt Y$.
Hence, if $f: \wt X \to \wt Y$ induces $f^*: \CI(\wt Y; \Ie \Lambda^1) \to \CI(\wt X; \Ie \Lambda^1)$, it will also induce a map
\begin{equation*}
	f^*: \CI(\wt Y; \Iie \Lambda^1) \to \CI(\wt X; \Iie \Lambda^1)
\end{equation*}
if $f^*(\rho_{\wt Y})$ is divisible by $\rho_{\wt X}$. That is, we want $f$ to map the boundary of $\wt X$ to the boundary of $\wt Y$ (a priori, it could map a boundary face of $\wt X$ onto all of $\wt Y$).
For maps $f$ coming from pre-stratified maps, this condition holds and hence the map $HS(f)$ induces a bounded map between iterated incomplete edge $L^2$ spaces.
Of course, once $f^*$ induces a map on $\Iie \Lambda^1$, it extends to a map on $\Iie \Lambda^*$.

\subsection{Stratified homotopy invariance of the analytic signature class}

Suppose we have a stratum-preserving smooth homotopy equivalance between stratified spaces $f: \hat X \to \hat Y$.
Recall that $X$ and $Y$ denote the regular parts of $\hat X$ and $\hat Y$, respectively. Recall the map $r: \hat Y \to B\Gamma$ 
and the 
 flat bundle $\cV '$ of finitely generated $C^*_r\Gamma$-modules over $\hat Y:$ 
 $$
 \cV '\,=\, C^*_r \Gamma \times_\Gamma r^*(E \Gamma).
 $$ Notice that using the blowdown map $\wt Y \rightarrow \hat Y,$  $\cV '$ induces a flat bundle, still denoted $\cV '$ on $\wt Y.$
Consider  $\cV = f^*\cV'$ the corresponding flat bundle over $\hat X$.  We have a flat connection on $\cV'$, $\nabla_{\cV'}$, over $Y$ (and $\wt Y$) and associated differential $d_{\cV'}$, and
 corresponding connection $\nabla_{\cV}$ and differential $d_{\cV}$ on $X$ (and $\wt X$).
It is straightforward to see that the Hilsum-Skandalis replacement of $f$ constructed above extends to 
\begin{equation*}
	HS(f): \CI( Y; \Iie \Lambda^* \otimes \cV') \to \CI( X; \Iie \Lambda^* \otimes \cV)
\end{equation*}
and induces a bounded operator between the corresponding $L^2$ spaces.

We now explain how the rest of the argument of Hilsum-Skandalis extends to this context.

Suppose $(f_t)_{0\leq t \leq 1}: \hat X \to \hat Y$ is a homotopy of stratum-preserving smooth homotopy equivalences, let $\bbD_{\wt Y}$ be as above.
Assume that $(e_s)_{0\leq s \leq 1}: \bbD_{\wt Y} \to \wt Y$ is a homotopy of smooth maps such that, for any $s \in [0,1]$,  $p_s = e_s \circ \bbD(f_s): f_s^*\bbD_{\wt Y} \to \wt Y$ induces a surjective map on $\iie$ vector fields.
Choose a smooth family of bundle isomorphisms (over $ \wt X$) $A_s: f_s^* \bbD_{\wt Y} \rightarrow f_0^* \bbD_{\wt Y}$ 
$( 0 \leq s \leq 1)$
such that $A_0= Id.$ Set $\cT_s = A_s^* \cT_0$ where $\cT_0$ is a Thom form for
the bundle $f_0^* \bbD_{\wt Y} \rightarrow \wt X.$
 Consider $\nabla$ be a flat unitary connection on $\cV '$. It induces an exterior derivative $d_{\cV '}$ on the bundle $\wedge^* T^* \wt Y \otimes \cV '.$
Choose a smooth family of $C^*_r \Gamma-$bundle isomorphism $U_s$ from 
the bundle $(p_s \circ A^{-1}_s)^* \cV ' \rightarrow f_0^* \bbD_{\wt Y}$ onto the bundle 
$p_0^*\cV ' \rightarrow f_0^* \bbD_{\wt Y}$ such that $U_0 =Id.$ Implicit in the statement of the next Lemma is the fact that,  for each $s \in [0, 1]$, $p_s \circ A^{-1}_s$ induces a morphism 
 from the space of sections of the bundle $ \cV ' \rightarrow \wt Y$ on
the space of sections of the bundle $(p_s \circ A^{-1}_s)^* \cV ' \rightarrow f_0^* \bbD_{ \wt Y}.$

\begin{lemma}\label{lem:Htpy}
Under the above hypotheses and notation,
there exists a bounded operator $\Upsilon: L^2_{\iie}(\wt Y; \Iie \Lambda^* \otimes \cV ') \to L^2_{\iie}(f_0^*\bbD_{\wt Y}; \Iie \Lambda^* \otimes p_0^*\cV ')$ such that 
\begin{equation*}
(  {\rm Id} \otimes U_1) \circ  (\,  \cT_0 \wedge (p_1 \circ A_1^{-1})^*\,) - (\cT_0 \wedge p_0^*) = p_0^*( d_{\cV '}) \Upsilon + \Upsilon d_{\cV '}.
\end{equation*}
\end{lemma}

\begin{proof}
We follow  Hilsum-Skandalis. Consider the map
$$
H: f_0^* \bbD_{ \wt Y} \times [0,1] \rightarrow \wt Y
$$
$$
(x,s) \rightarrow H(x,s)= p_s \circ A_s^{-1}(x).
$$
 Then the required map $\Upsilon$ is defined by, $ \forall \omega \in  L^2_{\iie}(\wt Y; \Iie \Lambda^* \otimes \cV '),$
 $$
 \Upsilon (\omega) = \int_0^1 i_{\frac{\partial}{\partial t}} \bigl( \,
  U_t \circ (p_t \circ A^{-1}_t)_F^* \otimes ( \cT_0 \wedge H^* \omega) \, \bigr) d t.
 $$
\end{proof}

We need to see how this construction handles composition.
Recall that given $f:\wt X \to \wt Y$ we are taking $\bbD_{\wt Y}$ to be the $\ie$ vectors over $\wt Y$ with length bounded by one,  $\bbD(f): f^*\bbD_{\wt Y} \to \bbD_{\wt Y}$ the natural map \eqref{PullBackDiag}, $e: \bbD_{\wt Y} \to {\wt Y}$ the exponential map, $p= e\circ \bbD(f)$, and $\cT$ a Thom form on $f^* \bbD_{\wt Y}$, and then
\begin{equation*}
	HS(f)u = (\pi_X)_* (\cT \wedge p^*u).
\end{equation*}

Now suppose $\wt X$, $\wt Y$, and $\wt Z$ are manifolds with corners and iterated fibration structures, and 
\begin{equation*}
	\wt X \xrightarrow{h} \wt  Y \xrightarrow{f} \wt Z
\end{equation*}
are smooth b-maps that send boundary hypersurfaces to boundary hypersurfaces and the fibers of boundary fibrations to the fibers of boundary fibrations. Assume that the map  $r: \hat X \rightarrow B \Gamma$ above is of the form 
$r=r_1 \circ f$ for a suitable map $r_1: \hat Z \rightarrow B \Gamma.$ We then get a flat $C^*_r\Gamma-$bundle 
$\cV ''$ over $\hat Z$ (and $\wt Z$) such that $\cV ' = f^* \cV ''.$
Denoting the various $\pi_{\cdot}$'s by $\tau$'s, we have the following diagram
\begin{equation*}
	\xymatrix{ (f \circ p')^* \bbD_{\wt Z} \ar[d]^{\tau_1} \ar[rd]^{\wt p} \ar@/_2pc/[dd]_{\tau} \ar@/^2pc/[rrdd]^{p''} \\
	h^*\bbD_{\wt Y} \ar[d]^{\tau_2} \ar[rd]^{p'} & f^*\bbD_{\wt Z} \ar[d]^{\tau_0} \ar[rd]^p \\
	\wt X \ar[r]^h & \wt Y \ar[r]^f & \wt Z }
\end{equation*}
where $\wt p(\zeta, \xi, \eta) = (p'(\zeta, \xi), \eta)$, and, with $\cT$ standing for a Thom form, we define
\begin{equation*}
\begin{gathered}
	HS(f): \CI(\wt Z; \Iie \Lambda^1 \otimes \cV '') \to \CI(\wt Y; \Iie \Lambda^1 \otimes \cV '), \\
	HS(h): \CI(\wt Y; \Iie \Lambda^1 \otimes \cV ' ) \to \CI(\wt X; \Iie \Lambda^1 \otimes \cV )\\
	HS(f,h): \CI(\wt Z; \Iie \Lambda^1 \otimes \cV '') \to \CI(\wt X; \Iie \Lambda^1\otimes \cV), \\ 
	HS(f)(u) = (\tau_0)_* (\cT_{\tau_0} \wedge p^*u), \quad HS(h)(u) = (\tau_2)_* (\cT_{\tau_2} \wedge (p')^*u) \\
	HS(f,h)(u) = \tau_* (\cT_{\tau_2} \wedge (\wt p)^* \cT_{\tau_0} \wedge (p'')^*u)
\end{gathered}
\end{equation*}

\begin{lemma}\label{lem:Comp}
 $HS(f,h) = HS(h) \circ HS(f)$ and $HS(f, h) - HS( f \circ h) = d_{\cV}\Upsilon + \Upsilon d_{\cV ''}$ for some bounded operator $\Upsilon$.
\end{lemma}

\begin{proof} For simplicity, we give the proof only in the case $\Gamma =\{1\}.$
Using the specific definitions of $\tau_1, \wt p, p', \tau_0$ one checks easily that $(\tau_1)_* \wt p^* = (p')^* (\tau_0)_*.$ 
Therefore, $(\wt p)^* \cT_{\tau_0}$ is indeed a Thom form associated with $\tau_1.$
Since $p''= p \circ \wt p,$ one gets:
$$
	HS(f,h)
	= (\tau_2)_* (\tau_1)_* (\cT_{\tau_2}  \wedge \wt p^*( \cT_{\tau_0} \wedge p^*) ) 
	$$ Then replacing  $(\tau_1)_* \wt p^*$ by $(p')^* (\tau_0)_*$  one gets:
	
	$$HS(f,g)= (\tau_2)_* (\cT_{\tau_2}  \wedge (p')^* ( (\tau_0)_*  (\cT_{\tau_0} \wedge (p)^*) ) )
	= HS(h) \circ HS(f).
	$$

Next, notice that the maps
\begin{equation*}
	(t; \zeta, \xi, \eta) \mapsto \exp_{f ( \exp_{h(\zeta)} (t\xi) )} (\eta)
\end{equation*}
are a homotopy between $p'': (f \circ p')^*\bbD_{\wt Z} \to \wt Z$ and $\hat p: (f \circ h)^*\bbD_{\wt Z} \to \wt Z$  
within submersions. Hence we can use the previous lemma to guarantee the existence of $\Upsilon$. 
\end{proof}

Instead of the usual $L^2$ inner product, we will consider the quadratic form
\begin{equation*}
\begin{gathered}
	Q_{\wt X}: \CI( \wt X ; \Iie \Lambda^* \otimes \cV ) \times \CI( \wt X ; \Iie \Lambda^* \otimes \cV ) \to C^*_r\Gamma \\
	Q_{\wt X}(u,v) = \int_{\wt X} u \wedge v^*
\end{gathered}
\end{equation*}
and also the analogous $Q_{\wt Y}$, $Q_{\bbD_{\wt Y}}$, $Q_{f^*\bbD_{\wt Y}} $. Recall that any element of $\CI( \wt X ; \Iie \Lambda^* \otimes \cV )$ vanishes at the boundary of $\wt X$ so that $Q_{\wt X}$ is indeed well defined. (We point out that the corresponding quadratic form in  Hilsum-Skandalis \cite[page 87]{H-S} is given by $ i^{ |u| (n- |u|)} Q_{\wt X}(u,v).$)
We denote the adjoint of an operator $T$ with respect to $Q_{\wt X}$ (or $Q_{\wt Y}$) by $T'$. Thus, for instance, $d_{\cV}' = - d_{\cV}$.

From Theorem \ref{theo:kk}, we know that the signature data on $\hat X$ defines an element of $K_{\dim X}(C^*_r\Gamma)$ and similarly for the data on $\hat Y$.
Hilsum and Skandalis gave a criterion for proving that two classes are the same which we now employ.

\begin{proposition}\label{prop:h-ska} Consider a stratum-preserving homotopy equivalence 
$f: \hat X \rightarrow \hat Y,$ where $\dim \hat X=n $ is even.  Denote still by $f$ the induced map $\wt X \rightarrow \wt Y.$
The bounded operator 
\begin{equation*}
	HS(f):L^2_{\iie}(\wt Y; \Iie \Lambda^* \otimes \cV') \to L^2_{\iie}(\wt X; \Iie \Lambda^* \otimes \cV)
\end{equation*}
satisfies the following properties:
\begin{itemize}
 \item [a)] $HS(f) d_{\cV'} = d_{\cV} HS(f)$ and $HS(f)(\Dom d_{\cV'}) \subseteq \Dom d_{\cV}$
 \item [b)] $HS(f)$ induces an isomorphism $HS(f): \ker d_{\cV'}/ \Im d_{\cV'} \to \ker d_{\cV} / \Im d_{\cV}$
 \item [c)] There is a bounded operator $\Upsilon$ on a Hilbert module associated to $Y$ 
 such that $\Upsilon(\Dom d_{\cV'}) \subseteq \Dom d_{\cV'}$ and
	$\Id - HS(f)'HS(f) = d_{\cV'} \Upsilon + \Upsilon d_{\cV'}$
 \item [d)]
	There is a bounded self-adjoint involution $\eps$ on $Y$ such that $\eps(\Dom d_{\cV'}) \subseteq \Dom d_{\cV'}$, 
	which commutes with $\Id - HS(f)'HS(f)$ and anti-commutes with $d_{\cV'}$.
\end{itemize}
Hence the signature data on $\hat X$ and $\hat Y$ define the same element of $K_{0}(C^*_r\Gamma)$. 
\end{proposition}

\begin{proof}
The final sentence follows from ($a$)-($d$) and Lemma 2.1 in Hilsum-Skandalis \cite{H-S}.

In \cite{ALMP1} we showed that the signature operator has a unique closed extension, it follows that so do $d_{\cV}$ and $d_{\cV'}$ (see, e.g., \cite[Proposition 11]{Hunsicker-Mazzeo}). Since this domain is the minimal domain, as soon as we know that an operator is bounded in $L^2_{\iie}$ and commutes or anticommutes with these operators, we know that it preserves their domains.

a) Since $HS(f)$ is made up of pull-back, push-forward, and exterior multiplication by a closed form, $HS(f)d_{\cV'} = d_{\cV}HS(f)$. 

b) From ($a$) we know that $HS(f)$ induces a map $\ker d_{\cV'}/ \Im d_{\cV'} \to \ker d_{\cV} / \Im d_{\cV}$.
Let $h$ denote a homotopy inverse of $f$ and consider
\begin{equation*}
	HS(h): L^2_{\iie}(\wt X; \Iie \Lambda^* \otimes \cV) \to L^2_{\iie}(\wt Y; \Iie \Lambda^* \otimes \cV').
\end{equation*}
We know from Lemma \ref{lem:Comp} that $HS(f \circ h)$ and $HS(h) \circ HS(f)$ induce the same map in cohomology and, from Lemma \ref{lem:Htpy} that $HS(f \circ h)$ induces the same map as the identity. Since the same is true for $HS(f \circ h)$ we conclude that $HS(h)$ and $HS(f)$ are inverse maps in cohomology and hence each is an isomorphism.

c) Recall that $p: f^* \bbD_{\wt Y} \rightarrow \wt Y$, being a proper submersion, is a fibration. Choose a Thom form 
$\wt \cT$ for the fibration $\pi_{\wt Y}: \bbD_{\wt Y} \rightarrow \wt Y$ so that $\bbD_{\wt Y} (f)^* \wt  \cT$ defines a Thom form 
for the fibration $\pi_{\wt X}: f^* \bbD_{\wt Y} \rightarrow \wt X.$
These two facts allow us to carry out the following computation, where $u \in C^\infty (\wt Y; \Iie \Lambda^* \otimes \cV')$ and 
$v \in C^\infty (\wt X; \Iie \Lambda^* \otimes \cV).$
\begin{equation*}
\begin{split}
	Q_{\wt X}(HS(f)u, v) 
	&= Q_{\wt X}\lrpar{ (\pi_{\wt X})_*( \bbD_{\wt Y} (f)^* \wt \cT \wedge p^*u), v } \\
	&= Q_{f^*\bbD_{\wt Y}} (\bbD_{\wt Y} (f)^* \wt \cT \wedge p^*u, \pi_{\wt X}^*v ) \\
	&= (-1)^{ n(n-| v |)} Q_{f^*\bbD_{\wt Y}} ( p^*u, \bbD_{\wt Y} (f)^* \wt \cT \wedge \pi_{\wt X}^*v ) \\
	&= (-1)^{ n(n-| v |)}Q_{\wt Y} ( u, p_* (\bbD_{\wt Y} (f)^* \wt \cT \wedge \pi_{\wt X}^*v ) ).
\end{split}
\end{equation*}
Since $n$ is even this shows that 
$HS(f)' v =  p_* (\bbD_{\wt Y} (f)^* \wt \cT \wedge \pi_{\wt X}^*v )$
and hence
\begin{equation*}
	HS(f)'HS(f)u = p_* ( \bbD_{\wt Y} (f)^* \wt \cT \wedge \pi_{\wt X}^*  (\pi_{\wt X})_*((\bbD_{\wt Y} (f)^* \wt \cT \wedge p^*u) ) ).
\end{equation*}
Next one checks easily that, for any differential form $ \omega $ on $\bbD_{\wt Y}$,
$$ \bbD_{\wt Y} (f)^*  \pi^*_{\wt Y}   (\pi_{\wt Y})_* \omega =  \pi^*_{\wt X}  (\pi_{\wt X})_* \bbD_{\wt Y} (f)^* \omega.
$$ 
and so, from the identity  $p^* = \bbD_{\wt Y} (f)^* e^*,$
\begin{equation*}
	HS(f)'HS(f)u = (e \circ \bbD_{\wt Y} (f))_* ( \bbD_{\wt Y} (f)^*(\,  \wt \cT \wedge \pi_{\wt Y}^*  (\pi_{\wt Y})_*( \wt \cT \wedge e^*u) ) ).
\end{equation*}

Now observe that  $\bbD_{\wt Y} (f): f^* \bbD_{\wt Y} \rightarrow  \bbD_{\wt Y},$ being an homotopy equivalence of manifolds with corners,
sends the relative fundamental class of $f^* \bbD_{\wt Y}$ to the relative fundamental class of $ \bbD_{\wt Y}$ and so
$$
Q_{f^* \bbD_{\wt Y}} ( \bbD_{\wt Y} (f)^*\alpha  , \bbD_{\wt Y} (f)^* \beta ) =Q_{ \bbD_{\wt Y}} (\alpha  ,  \beta ).
$$ From this identity, the previous equation, and the fact that $e$ induces a fibration, one checks easily that
$$
 Q_{\wt Y} ( HS(f)'HS(f)u , w)= Q_{\wt Y} ( e_* (\, \wt \cT \wedge \pi_{\wt Y}^*  (\pi_{\wt Y})_*( \wt \cT \wedge e^*u) ) , w)
$$
and hence
$$
HS(f)'HS(f)u = e_* (\,\wt \cT \wedge  \pi_{\wt Y}^*  (\pi_{\wt Y})_*( \wt \cT \wedge e^*u) ).
$$
Finally, $e$ is homotopic to $\pi_{\wt Y}$, and since
$$
(\pi_{\wt Y})_* (\, \wt \cT \wedge \pi_{\wt Y}^*  (\pi_{\wt Y})_*( \wt \cT \wedge \pi_{\wt Y}^*u) )= (\pi_{\wt Y})_* (\wt \cT \wedge \pi_{\wt Y}^*u) = u,
$$
Lemma \ref{lem:Htpy}, $\Id - HS(f)'HS(f) = d_{\cV'} \Upsilon + \Upsilon d_{\cV'}$ as required.

d) It suffices to take $\eps u = (-1)^{|u|} u$.
\end{proof}

\noindent {\bf Remark.} Consider now the case of   an odd dimensional Witt space $ \hat{X} $ endowed with a conic iterated metric $g$ and a reference
map $r: \hat{X} \rightarrow B \Gamma.$ We have defined in Part I  the higher signature index class 
$ \Ind\, (\wt \eth_{\sign}) \in KK_1(\bbC, C^*_r \Gamma)\simeq K_1(C^*_r \Gamma)  $ associated to the twisted signature 
operator 
defined by the data $ (\hat{X} , g, r) $. Recall that there is a suspension  isomorphism $\Sigma: K_1(C^*_r \Gamma) \leftrightarrow \wt K _0(C^*_r \Gamma \otimes C(S^1))$ which is induced by taking the
Kasparov product with the Dirac operator of $S^1$. Consider 
the even dimensional Witt space $\hat{X} \times S^1$ endowed with the obvious stratification and with  the reference map 
$$
 r \times \Id_{S^1}: \hat{X} \times S^1 \rightarrow B ( \Gamma \times \mathbb{Z}) \simeq B \Gamma \times S^1.
$$ 
As explained in \cite[p. 624]{LLP}, \cite[\S 3.2]{LPJFA}, the suspension of  the odd index class 
$ \Ind\, (\wt \eth_{\sign}) \in KK_1(\bbC, C^*_r \Gamma)\simeq K_1(C^*_r \Gamma) $
is equal to the even signature index class associated to the data
$ (\hat{X} \times S^1, g \times (d \theta)^2 , r  \times \Id_{S^1} ) .$ 
If now $f:  \hat{X}\to   \hat{Y} $  is  a stratified homotopy equivalence of odd dimensional
Witt spaces, then $f$ induces a stratified homotopy equivalence from $\hat{X}\times S^1$ to    
$\hat{Y}\times S^1$. By the previous Proposition the signature index classes of
$\hat{X}\times S^1$ and     
$\hat{Y}\times S^1$ are the same. Then using the suspension isomorphism $\Sigma$,  we deduce finally that the odd 
signature index classes associated to $\hat{X}$ and $\hat{Y} $ are the same.
Thus, the (smooth) stratified homotopy invariance of the signature index class is established
for Witt spaces of arbitrary dimension.

\section{Assembly map and stratified homotopy invariance of  higher signatures}
Consider  the assembly map $\beta: K_* (B\Gamma) \to K_*(C^*_r \Gamma)$.
The rationally injectivity  of this map is known as the  strong Novikov conjecture for $\Gamma$. In the closed case
it implies that the Novikov higher signatures are oriented homotopy invariants. The rational injectivity
of the assembly map is still unsettled in general, although it is known to hold for large classes
of discrete groups; for closed manifolds having these fundamental  groups the higher signatures are thus homotopy invariants.
The following  is the main result of this second paper:

\begin{theorem}\label{theo:main}
Let $\hat X$ be an oriented Witt space,  $r: \hat X\to B\pi_1 (\hat X)$ the classifying map
for the universal cover, and let $\Gamma:= \pi_1 (\hat X)$.
If the assembly map $K_* (B\Gamma) \to K_*(C^*_r \Gamma)$
is rationally injective, then the Witt-Novikov  higher signatures
$$\{ <\alpha, r_* L_* (\hat X)>, \alpha\in H^* (B\Gamma,\bbQ) \}$$
are stratified homotopy invariants
\end{theorem}
 
 \begin{proof}
The proof proceeds in four steps and is directly inspired by Kasparov's proof in the closed case, see
for example \cite{Kasparov-contemporary}  and the references therein :

   \begin{enumerate}
  \item Consider  $({\hat X}', r':{\hat X}'\to B\Gamma)$
 and $(\hat X,r: \hat X\to B\Gamma)$, with $r=r'\circ f$ and  $f: \hat X\to {\hat X}'$
 a stratified homotopy equivalence between (smoothly stratified) oriented Witt spaces. Denote by $ \wt \eth_{\sign}' $ the twisted signature operator associated to
 $({\hat X}', r':{\hat X}'\to B\Gamma)$.
 We have proved that
 $$ \Ind (\wt \eth_{\sign}) = \Ind (\wt \eth_{\sign}')  \quad \text{in}\quad K_* (C^*_r \Gamma)\otimes \bbQ  \,.$$
 \item We know that 
 the assembly map sends $r_* [\eth_{\sign}]\in K_* (B\Gamma)$ to the Witt index class $\Ind (\wt \eth_{\sign}).$
 More explicitly:
 $$\beta (r_* [\eth_{\sign}])= \Ind (\wt \eth_{\sign})\quad\text{in}\quad K_*(C^*_r \Gamma)\otimes \bbQ$$
  \item We deduce from the assumed rational injectivity of the assembly map  that
 $$ r_* [\eth_{\sign}] = (r')_* [\eth_{\sign}']\quad \text{in}\quad K_* (B\Gamma)\otimes \bbQ.$$
 \item Since we know from Cheeger/Moscovici-Wu that $\Ch_* (r_* [\eth_{\sign}])=r_* (L_* (\hat X))$
 in $ H_* (B\Gamma,\bbQ)$ we finally get that 
 $$ r_* (L_* (\hat X)) = (r')_* (L_* ({\hat X}'))\quad \text{in}\quad H_* (B\Gamma,\bbQ)$$
 which obviously implies the stratified homotopy invariance of  the higher signatures
 $\{ <\alpha, r_* L_* (\hat X)>, \alpha\in H^* (B\Gamma,\bbQ) \}$.  
\end{enumerate}
\end{proof}

Examples of discrete groups for which the assembly map is rational injective include:
amenable groups,
discrete subgroups of  Lie groups with a finite number of connected components, Gromov hyperbolic groups, discrete groups acting properly
on bolic spaces, 
countable subgroups of $GL(\bbK)$ for $\bbK$ a field.

\section{The symmetric signature on Witt spaces}

\subsection{The symmetric signature in the closed case}

Let $X$ be a closed orientable manifold and let $r:X\to B\Gamma$ be a classifying map
for the universal cover.
The symmetric signature of Mishchenko, $\sigma (X,r)$,
is a purely topological object \cite{Mish1}. In its most
sophisticated presentation, it  is an element in  the $L$-theory groups $L^* (\bbZ\Gamma)$.
 In general one can define the symmetric signature of any algebraic Poincar\'e complex, i.e.
a  cochain complex of finitely generated $\bbZ\Gamma$-modules satisfying a kind
of Poincar\'e duality. The Mishchenko symmetric signature corresponds to the choice
of the Poincar\'e complex defined by the cochains on the universal cover.
In the treatment of the Novikov conjecture one is in fact interested in a 
less sophisticated invariant, namely the image of $\sigma (X,r)\in L^* (\bbZ\Gamma)$
 under the natural map $\beta_{\bbZ}: L^* (\bbZ\Gamma)\to
L^* (C^*_r\Gamma)$. Recall also that there is a natural isomorphism $\nu: L^* (C^*_r\Gamma)\rightarrow
K_* (C^*_r \Gamma)$ (which is in fact valid for any $C^*$-algebra). The $C^*$-algebraic symmetric
signature is, by definition, the element $\sigma_{_{C^*_r \Gamma}}(X,r) := \nu (\beta_{\bbZ} (\sigma (X,r))$;
thus $\sigma_{_{C^*_r \Gamma}}(X,r)\in K_* (C^*_r \Gamma)$.
The following result, due to Mishchenko and Kasparov, generalizes the equality between the numeric 
index of the signature operator and the topological signature. With the usual notation:
\begin{equation}\label{equality}
\Ind (\wt \eth_{\sign})=\sigma_{_{C^*_r \Gamma}}(X,r)\in K_* (C^*_r \Gamma)
\end{equation}
As a Corollary we see that the signature index class 
is a homotopy invariant; this is the topological approach
to the homotopy invariance of the signature index class that we have mentioned in
the introductory remarks in  Section \ref{sect:shi}. The equality of the $C^*$-algebraic symmetric signature with the 
signature index class (formula (\ref{equality}) above) can be
restated as saying that the following diagram is commutative
\begin{equation}
  \label{diagram}
  \begin{CD}
   \Omega^{{\rm SO}}_* (B\Gamma) @>{{\rm Index}}>> K_i (C^*_r \Gamma)\\
    @VV{\sigma}V  @V{\nu^{-1}}VV\\
  L^* (\bbZ\Gamma)  @>{\beta_{\bbZ}}>> L^*(C^*_r \Gamma).
  \end{CD}
\end{equation}
where $i\equiv *$ mod 2.

\subsection{The symmetric signature on Witt spaces.}

The middle perversity intersection homology groups of a Witt space do satisfy
Poincar\'e duality over the rationals. Thus, it is natural to expect that for  a Witt space
$\hat X$ endowed with a reference map $r:\hat X\to B\Gamma$ it should be possible
to define a symmetric signature $\sigma^{{\rm Witt}}_{\bbQ \Gamma}(X,r) \in L^* (\bbQ \Gamma).$
And indeed,  the definition of symmetric signature
in the Witt context, together with its expected properties, such as Witt bordism invariance,
does appear in the literature, see for example  \cite{Wei},  \cite{Chang},  \cite{Wei-higher-rho}.

However, no rigorous account of this definition has appeared so far,
which is unfortunate given that things are somewhat more complicated than in the smooth
case.
First of all, 
in the first paper of Goresky-MacPherson \cite{GM} the intersection
product is not defined at the level of intersection chains; one needs to pass to
intersection homology. This means that we cannot reproduce the
definition of Mishchenko. 
Note, incidentally, that Greg Friedman has recently written an interesting article
\cite{friedman2}
on the chain-level description of the intersection product in intersection homology. This is based on previous work of Jim
McClure and might very well allow for a rigorous definition, \`a la Mishchenko,  of the symmetric signature in the
Witt context.

The Poincar\'e duality isomorphism for intersection homology was given a fresh 
treatment in the second paper of Goresky-MacPherson \cite{GM2}. 
One could hope to use the self-duality of the intersection chain sheaf
in order to induce a structure of Algebraic Poincar\'e Complex on the geometric intersection chain complex.
In other words, the Algebraic Poincar\'e Complex structure is given first to the complex of global sections of the sheaf and then pulled back to
the actual geometric intersection chain complex  (defined by a good triangulation) using suitable chain homotopies. This is probably the
idea  which is behind the above references.
Now, even this approach, although intuitively clear, needs careful explanation. The problem is, first of all, that by going to sheaves
and their global sections, we lose the finitely generated projective property of our modules. Second, and more
importantly,  all of \cite{GM2} takes place in the derived category, and the existence of explicit chain homotopies from
the sheaf-picture to the geometric-picture, is far from being obvious.

Summarizing: the existence of an explicit structure of  Algebraic Poincar\'e Complex  on the geometric intersection chain complex with values in the local system
defined by the Mishchenko bundle $r^* E\Gamma\times_\Gamma \bbQ\Gamma$ does not seem to be  obvious.

Fortunately, in a recent paper Markus Banagl \cite{banagl-msri} has given an alternative definition
of the symmetric signature on Witt spaces\footnote{Banagl actually concentrates on the more restrictive class of IP spaces, for which an integral symmetric signature,
i.e. an element in $L^* (\bbZ\Gamma)$, exists; it is easy to realize that his construction can be given for the larger
class of Witt spaces, producing, however, an element in $L^* (\bbQ \Gamma)$.} using  surgery techniques as well as previous results
of Eppelmann \cite{Eppelmann}. Banagl's symmetric signature is an element $\sigma^{\rm Witt}_{\bbQ\Gamma} (\hat X,r)\in L^* (\bbQ \Gamma)$;
we refer directly to Banagl's interesting article for the definition and only point out that directly from his construction we can
conclude that

\begin{itemize}
\item the symmetric signature $\sigma^{\rm Witt}_{\bbQ\Gamma}  (\hat X,r)$ is equal to (the rational) Mishchenko's symmetric signature
if $\hat X$ is a closed compact manifold;
\item the Witt symmetric signature is a Witt bordism invariant; it defines a group homomorphism
$\sigma^{\rm Witt} : \Omega^{\rm Witt}_* (B\Gamma) \to L^* (\bbQ \Gamma)$.
\end{itemize}

On the other hand, it is not known whether Banagl's symmetric signature $\sigma^{\rm Witt}_{\bbQ\Gamma} (\hat X,r)$
 is a stratified homotopy invariant.

We define the $C^*$-algebraic Witt symmetric signature as the image of $\sigma^{\rm Witt}_{\bbQ}  (\hat X,r)$
under the composite 
$$L^* (\bbQ \Gamma)\xrightarrow{\beta_{\bbQ}} L^* (C^*_r \Gamma)\xrightarrow{\nu} K_* (C^*_r \Gamma)\,.$$
We denote the $C^*$-algebraic Witt symmetric signature by $\sigma^{{\rm Witt}}_{_{C^*_r \Gamma}}(X,r)$.

\subsection{Rational equality of the Witt symmetric signature and of the signature index class}

Our most general goal would be to prove that
there is a commutative diagram
\begin{equation}
  \label{diagram-witt}
  \begin{CD}
   \Omega^{{\rm Witt}}_*(B\Gamma) @>{{\rm Index}}>> K_i (C^*_r \Gamma)\\
    @VV{\sigma^{\rm Witt}}V  @V{\nu^{-1}}VV\\
  L^* (\bbQ\Gamma)  @>{\beta_{\bbQ}}>> L^*(C^*_r \Gamma).
  \end{CD}
\end{equation}
or, in formul\ae
$$\sigma^{{\rm Witt}}_{_{C^*_r \Gamma}}(X,r)= \Ind (\wt \eth_{\sign}) \quad \text{in} \quad  K_i (C^*_r \Gamma)$$
with $\Ind(\wt \eth_{\sign})$ the signature index class decribed in our first paper.
We shall be happy with a little less, namely the rational equality.

\begin{proposition}\label{prop:magic}
Let $\sigma^{{\rm Witt}}_{C^*_r \Gamma}(X,r)_{\bbQ}$ and $ \Ind (\wt \eth_{\sign})_{\bbQ} $ be
the rational classes,  in the rationalized K-group $K_i (C^*_r \Gamma)\otimes \bbQ$, defined by the Witt symmetric signature 
and by the signature index class.
Then 
\begin{equation}\label{rational}
\sigma^{{\rm Witt}}_{C^*_r \Gamma}(X,r)_{\bbQ}= \Ind (\wt \eth_{\sign})_{\bbQ}\quad\text{in}\quad K_i (C^*_r \Gamma)\otimes \bbQ
\end{equation}
\end{proposition}
\begin{proof}
We already  know from \cite{banagl-msri} that the rationalized symmetric signature 
defines  a homomorphism from $(\Omega^{{\rm Witt}}_*(B\Gamma))_{\bbQ}$ to $K_i (C^*_r \Gamma)\otimes \bbQ$.
However, it also clearly defines a homomorphism $(\Omega^{{\rm Witt},s}_*(B\Gamma))_{\bbQ}\to  K_i (C^*_r \Gamma)\otimes \bbQ$,
exactly as the signature index class.
For notational convenience, let $\mathcal{I}: (\Omega^{{\rm Witt},s}_*(B\Gamma))_{\bbQ}\to  K_i (C^*_r \Gamma)\otimes \bbQ$
be the (Witt) signature  index morphism; let $\mathcal{I}': (\Omega^{{\rm Witt},s}_*(B\Gamma))_{\bbQ}\to  K_i (C^*_r \Gamma)\otimes \bbQ$
be the (Witt) symmetric signature morphism. We want to show that 
$$\mathcal{I} = \mathcal{I} '\,.$$
We know from Proposition \ref{prop:ko}  that the natural map 
$\Omega^{{\rm SO}}_*(B\Gamma) \rightarrow  \Omega^{{\rm Witt},s}_*(B\Gamma)$ induces   a rational surjection
$$  s:  ( \Omega^{{\rm SO}}_*(B\Gamma) )_{\bbQ}\rightarrow  (\Omega^{{\rm Witt},s}_*(B\Gamma))_{\bbQ}. $$
In words, a  smoothly stratified Witt space $X$ with reference map $r:X\to B\Gamma$ is smoothly stratified
Witt bordant to $k$-copies of a closed oriented
compact manifold $M$ with reference map $\rho: M\to B\Gamma$.
Moreover, we remark that the Witt  index classes and the Witt symmetric signature of an oriented  closed compact manifold
coincide with the classic signature index class and the Mishchenko symmetric signature. 
Then
$$\mathcal{I} ([X,r])= \mathcal{I} (k[M,\rho])=\mathcal{I}' (k[M,\rho]) = \mathcal{I}' ([X,r])\,$$
with the first and third equality following from the above remark and the second equality obtained using 
the fundamental result of Kasparov and Mishchenko on closed manifolds.
The proof is complete.
\end{proof}

The above Proposition  together with Proposition \ref{prop:h-ska} implies at once the following result:

\begin{corollary}\label{cor:shi-of-banagl}
The $C^*$-algebraic symmetric signature defined by Banagl is a rational stratified 
homotopy invariant.
\end{corollary}

This Corollary does not seem to be obvious from a  purely topological point of view.

\section{Epilogue}

Let $\hat X$ be an orientable Witt pseudomanifold with fundamental group $\Gamma$. 
We endow the regular part of $\hat X$ with an iterated conic metric $g$.
Let $\hat X ^\prime$ 
be a Galois $\Gamma$-covering and $r: \hat X\to B\Gamma$ a classifying map for $ \hat X ^\prime$. 
We now restate once more  the signature package for the pair $(\hat X,r:\hat X\to B\Gamma)$ indicating precisely where 
the individual items have been established in our two papers.

\begin{list}
 {(1)} \item The signature operator defined by the iterated conic metric $g$
with values in the Mischenko bundle $r^* E\Gamma \times_\Gamma C^*_r\Gamma$
defines  a signature index class  $\Ind (\wt \eth_{\sign})\in K_* (C^*_r \Gamma)$, $* \equiv \dim X \;{\rm mod}\; 2$.\\
{\it This was established in Theorem 7.6 of Part I.}
 \end{list}

\begin{list}
 {(2)} \item The signature index class is a (smooth) Witt bordism invariant; more precisely  it defines a group homomorphism 
$\Omega^{{\rm Witt},s}_* (B\Gamma) \to K_* (C^*_r \Gamma)\otimes \bbQ$.\\
{\it This is Theorem \ref{thm:bordism} in this paper, together with \eqref{bordsim-inv}.}
\end{list}

\begin{list}
 {(3)}\item The signature 
index class is a stratified homotopy invariant.\\
{\it This is Proposition \ref{prop:h-ska} in this paper.}
\end{list}

\begin{list}
 {(4)}
\item  There is a  K-homology signature class $[\eth_{\sign}]\in K_* (X)$ whose Chern
 character is, rationally, the  homology L-Class of Goresky-MacPherson.\\
  {\it This is Theorem 7.2 in Part I and Theorem \ref{theo:cmwu} of this paper.}
  \end{list}

 \begin{list}
 {(5)}
\item The assembly map $\beta: K_* (B\Gamma)\to K_* (C^*_r\Gamma)$ sends the
class $ r_* [\eth_{\sign}]$ into $\Ind (\wt \eth_{\sign})$.\\
{\it This is Corollary 7.7 in Part I.}
\end{list}

\begin{list}
 {(6)}
\item If  the assembly map  is rationally injective one  can deduce from the above results the homotopy invariance of the
 Witt-Novikov higher signatures.\\{\it This is Theorem \ref{theo:main} in this paper.}
\end{list}

 \begin{list}
 {(7)} \item There is a topologically defined $C^*$-algebraic symmetric signature 
 $\sigma^{{\rm Witt}}_{C^*_r \Gamma}(X,r)$\\$\in  K_* (C^*_r \Gamma)$ which is equal to the analytic index class
 $ \Ind (\wt \eth_{\sign})$ rationally.\\
{\it  This is Banagl's construction together with our Proposition \ref{prop:magic} in this paper.}
   \end{list}

\end{document}